\documentclass{article}

\usepackage{graphicx}
\usepackage{amsmath,amssymb}
\usepackage{amsfonts}
\usepackage{amstext}
\usepackage{amsthm}
\usepackage{amsopn}
\usepackage{amsbsy}
\usepackage{color}
\usepackage{url}
\usepackage{enumerate}
\usepackage{geometry}

\newcommand{\p}{\mathbf p}
\newcommand{\x}{\mathbf x}
\newcommand{\m}{\mu}
\newcommand{\ds}{\displaystyle}
\newcommand{\ii}{\infty}
\newcommand{\e}{\varepsilon}
\newcommand{\s}{\sigma}
\newcommand{\pa}{\partial}

\newtheorem{definition}{Definition}
\newtheorem{theorem}{Theorem}
\newtheorem{corollary}{Corollary}
\newtheorem{remark}{Remark}
\newtheorem{lemma}{Lemma}
\newtheorem{example}{Example}

\begin{document}

\title{Trends to equilibrium for a class of relativistic diffusions}

\author{J\"urgen Angst \\  \small{IRMAR, Universit\'e Rennes 1, Campus de Beaulieu, 35042 Rennes Cedex, France} \\
\small{\url{jurgen.angst@univ-rennes1.fr}}}

\maketitle

\begin{abstract}
A large class $\mathcal C$ of relativistic diffusions with values in the phase-space of special relativity was introduced in \cite{angst} in order to answer some open questions concerning the asymptotic behaviour of two examples of such processes \cite{dr,dunkel1,dunkel2}. 
In particular, the equilibrium measures of these diffusions were explicitly computed, and their hydrodynamic limit was shown to be Brownian. In this paper, we address the question of the trends to equilibrium of the diffusions of the whole class $\mathcal C$. We show the existence of a spectral gap using the method introduced in \cite{bakry} and deduce the exponential decay of the distance to equilibrium  in $\mathbb L^2-$norm and in total variation. A similar result was obtained recently in \cite{simone2} for a particular process of the class $\mathcal C$. \par
\end{abstract}

\section{Introduction}
\label{intro}
The study of stochastic processes in the framework of special relativity goes back to the 1960s and the pionneering work of Dudley \cite{dudley1,dudley2}. Since the late 1990s, there has been a renewed interest in the subject  with the work of Debbasch and his co-authors on the Relativistic Ornstein-Uhlenbeck Process \cite{bdr3,bdr1,dmr,dr}, then those of Dunkel and H\"anggi on the so-called ``relativistic Brownian motion'' \cite{dunkel1,dunkel2}. The notion of relativistic diffusion has been extented to the realm of general relativity in \cite{flj,deb} and the literature on the topic is now thriving, both in Mathematics, see for example \cite{flj2,nonexplo,angstannIHP,angstpoisson} and in Physics
 \cite{haba,hermann,chevalier} etc. 
\bigskip

In this article, we address the question of the trends to equilibrium for a large class $\mathcal C$ of relativistic diffusions with values in the phase-space of special relativity, \emph{i.e.} the unitary tangent bundle of  Minkowski space-time. This class of processes was introduced in \cite{angst} in order to answer several open questions concerning the long-time asymptotic behavior of two examples of such diffusions, namely the ones considered in \cite{dmr,dunkel1,dunkel2}. To our knowledge, the class $\mathcal C$ includes most of the Minkowskian diffusions introduced in the physical literature.
\bigskip

We show that under the same mild hypotheses as in \cite{angst}, and for all diffusions of the class $\mathcal C$, the equilibrium measure of the ``momentum subdiffusion'' satisfies a Poincar\'e inequality and we thus deduce that the rate of convergence to equilibrium is exponential both in $\mathbb L^2-$norm and in total variation. A similar result was obtained recently in \cite{simone2} for a particular process of the class $\mathcal C$: the relativistic diffusion process associated to the kinetic relativistic Fokker-Planck equation considered in \cite{simone1,dunkel1,dunkel2,haba}. The method we follow here is the one developped in \cite{bakry,bakry2}, which generalize the classical Bakry-\'Emery criterion when the potential associated to the equilibrium measure is not strictly convex. It is based on the existence of a Lyapounov function associated to the infinitesimal generator of the diffusion.
\bigskip

The structure of the article is the following: in the next section, we introduce some notations and we recall the definition of the class $\mathcal C$ of relativistic diffusions considered in the sequel. In Section \ref{sec.results}, we state our results concerning the trends to equilibrium of the momentum components of the diffusions. The last section \ref{sec.proof} is devoted to the proof of our main result, namely the Poincar\'e inequality satisfied by the equilibrium measure. 

\section{The class $\mathcal C$ of relativistic diffusions} \label{sec.rappels}
Fix $d \geq 1$ an integer, and denote by $||\x|| = \sqrt{|x^1|^2 + \ldots + |x^d|^2}$ the Euclidian norm of a vector $\x \in \mathbb R^d$.
Let $\mathbb R^{1,d}$ denote the Minkowski space of special relativity. In its canonical basis, denote by $x = (x^{\m}) = (x^0, \: x^i) = (x^0, \:\x)$ the coordinates of the generic point, with greek indices running $0,..,d$ and latin indices running $1,..,d$.  The Minkowskian pseudo-metric is thus given by  
$$ds^2 = |dx^0|^2 - \sum_{i=1}^d\limits |dx^i|^2.$$

The world line of a particle with positive mass $m$ is a timelike path in $\mathbb R^{1,d}$, which we can always parametrize by its arc-length, or proper time $s$. So the moves of such particle are described by a path $s\mapsto (x^\mu_s)$ in $\mathbb R^{1,d}$, having momentum $p = (p_s)$ given by $ p= (p^{\m}) = (p^0,  p^i) = (p^0,  \p) $, where $d p^{\m}_s := m \,d x^{\m}_s/d s$, 
and satisfying the pseudo-norm relation 
$$|p^{0}|^2 - ||\p||^2 = m^2. $$
 We shall consider here future directed world lines of type $(t,\x_t)_{t \geq 0}$, and take $m=1$. Introducing the velocity $\mathbf v=(v^1, \ldots,  v^d)$ by setting $ v^{i} := d x^{i}/d t$, and working with the usual spherical coordinates $(r,\theta) \in \mathbb R_+ \times \mathbb S^{d-1}$,  $r := || \p ||$ and $\theta := \p/ r=: \left( \theta^1, \ldots,  \theta^d \right) $, we get at once: 
$$\left \lbrace \begin{array}{ll}
\ds{ p^0} & \ds{ = \frac{dt}{ds} = \sqrt{1+ r^2} =\left( 1- |\mathbf v|^2 \right)^{-1/2},} \\
\\
\ds{p} & \ds{ = \sqrt{1+ r^2} (1,  \mathbf v)}.
\end{array}
\right.$$

Thus, a full space-time trajectory $ (x_t, p_t)_{t \geq 0}= (t, \x_t, p^0_t, \p_t)_{t \geq 0} $
which takes values in the positive part of the unitary tangent bundle $T^1_+ \mathbb R^{1,d}$, is determined by the mere knowledge of its spacial components $(\x_t,\p_t)$.  We can therefore, from now on, focus on spacial trajectories $ t\mapsto (\x_t,\p_t)$ which take values in the Euclidian product $\mathbb R^d \times  \mathbb R^d $. 

\subsection{Definition of the class $\mathcal C$}

Let us first recall the definition of the class $\mathcal C$ introduced in \cite{angst}.
\begin{definition}
The relativistic diffusions of the class $\mathcal C$ are the processes of type $(x_t, p_t)_{t \geq 0}= (t, \x_t, p^0_t, \p_t)_{t \geq 0}$ in $T^1_+ \mathbb R^{1,d}$, where the associated spatial process $(\x_t,  \p_t)_{t \geq 0}$ is itself a diffusion, solution of a stochastic differential system of the form, for $1 \leq i \leq d$:
$$ (\star)\; \left \lbrace \begin{array}{lr}  d x_t^i  = f(r_t)  p^i_t  d t  \\ \vspace{-3mm} \\ 
d p_t^i = - b(r_t)  p^i_t  d t  + \sigma(r_t)\Big(\beta [1+\eta(r_t)^2]\Big)^{-1/2} [dW_t^i + \eta(r_t)\theta_t^i dw_t] & \end{array} \right., $$
where the real functions $f$, $b$, $\sigma$, $\eta$ are continuous on $\mathbb R_+$ and satisfy the following hypotheses,
for some fixed $\e >0$:   
$$({\bf{\mathcal H}}) \quad \left \lbrace
 \begin{array}{l}
\ds{\sigma \geq \e \;\hbox{on} \; \mathbb R_+ \; ;  \;\;  g(r):= \frac{2r b(r)}{\sigma^2(r)} \geq \e \; \hbox{ for large } r\; ;} \\
\ds{ \lim_{r\to\ii} \, e^{-\e' r} f(r) = 0\;  \hbox{ for some }  \e'<  \beta \e/2} .
 \end{array} \right.$$ 
\end{definition} 
In the definition above, $\mathbf W :=(W^1, \ldots , W^d)$ denotes a standard $d$-dimensional Euclidian Brownian motion, $w$ denotes a standard real Brownian motion, independent of $\mathbf W$, and $\beta >0$ is an inverse heat parameter.

\begin{example}
In the simplest case of constant functions $f$, $b$, $\s$, and $\eta=0$, the process $(\x_t)_{t \geq 0}$ is an integrated Ornstein-Uhlenbeck process. The process considered by Debbasch et al. in \cite{bdr3,bdr1,dmr,dr}, they call Relativistic Ornstein-Uhlenbeck Process (ROUP), corresponds to: 
$$ f(r) = b(r) = (1+r^2)^{-1/2} ,\; \; \s(r) = \sqrt{2} , \;\;\eta =0, $$
and the process considered by Dunkel and H\"anggi \cite{dunkel1,dunkel2} corresponds to: 
$$ \left \lbrace \begin{array}{l}
f(r) = (1+r^2)^{-1/2} ,\; \; b(r) = 1- {d  \beta^{-1} (1+r^2)^{-1/2}} , \\
\\
 \s(r) = \sqrt{2\sqrt{1+r^2}} ,  \;\; \eta(r) = r .
 \end{array} \right.$$ 
\end{example}

\subsection{Infinitesimal generator of the momentum diffusion}
If $(x_t, p_t)_{t \geq 0}= (t, \x_t, p^0_t, \p_t)_{t \geq 0}$ is a relativistic diffusions of the class $\mathcal C$, then the process $(\p_t)_{t \geq 0}$ is itself a diffusion process, we will call the \emph{momentum diffusion}. In spherical coordinates $\p=(r, \theta) \in \mathbb R_+ \times \mathbb S^{d-1}$, its infinitesimal generator is given by
\begin{equation} \label{eq.gener1}  \mathcal L_{\s^2} :=\, \mathcal L_r + \frac{\s^2(r)}{2\,\beta\,r^2}\: \Delta_{\mathbb S^{d-1}}, \end{equation}
where $\Delta_{\mathbb S^{d-1}}$ denotes the usual spherical Laplacian on $\mathbb S^{d-1}$ and $\mathcal L_r $ is the infinitesimal generator of the \emph{radial} process $r_t:=||\p_t||$:
$$\mathcal L_r :=\frac{\s^2(r)}{2\,\beta} \,\bigg( \pa_r^2 + \frac{d-1}{r} \partial_r  - \bigg[ \frac{d-1}{r}\times \frac{\eta^2(r)}{1+\eta(r)^2} + \beta\, g(r) \bigg] \partial_r \bigg) . $$
Let us use the same notations as in \cite{angst}, that is:
$$ \m(r):=\exp \left(   \int_1^r  \frac{d \rho}{\rho(1+\eta(\rho)^2)} d\rho\right), \quad G(r) := \int_0^r g(\rho)d \rho,$$
and introduce the functions $V : \mathbb R_+ \to \mathbb R$ and $U : \mathbb R^d \to \mathbb R$: 
$$V(r) := \int_1^r \frac{d-1}{\rho}\times \frac{\eta^2(\rho)}{1+\eta(\rho)^2} d\rho  + \beta G(r), \quad U(\p) := V(|| \p||).$$
Then, if $\Delta$ and $\nabla$ are the usual Laplacian and gradient in $\mathbb R^d$, the generator $\mathcal L_{\sigma^2}$ at $\p$ can be re-written under the familiar form: 
\begin{equation} \label{eq.gener2}  \mathcal L_{\s^2} := \frac{\s^2(||\p||)}{2 \beta} \times \mathcal L, \;\; \hbox{with} \;\; \mathcal L:= \Delta - \nabla U(\p). \nabla \;\;.\end{equation}
In the sequel, we denote by $\Gamma$ the operator ``carr\'e du champ'' associated to $\mathcal L$, that is for good functions $f$ and $g$ in the domain of $\mathcal L$:
\begin{equation}  \Gamma(f,g) : = \mathcal L(fg)-  \mathcal L(f)g -f\mathcal L(g).\label{eq.carre} \end{equation}

\section{Statement of the results} \label{sec.results}
We can now state our results concerning the trends to equilibrium of the momentum diffusion. 
In the following lemma, we explicit the invariant (or equilibrium) measure of the process $(\mathbf p_t)_{t \geq 0}$.
 
\begin{lemma}\label{lem.inv}
Let $(t, \mathbf x_t, p^0_t, \mathbf p_t)_{t \geq 0}$ a diffusion of the class $\mathcal C$. Then, the momentum subdiffusion $(\mathbf p_t)_{t \geq 0}$ with values in $\mathbb R^d$ is ergodic and its invariant probability measure is given by: 
\begin{equation}
\nu(\mathbf p) := \frac{1}{Z} \times \frac{e^{- U( \mathbf p) }}{\sigma^2(||\mathbf p||)} d \mathbf p,
\end{equation}
where $d \mathbf p$ denote the Lebesgue measure in $\mathbb R^d$ and $Z$ is a normalizing constant.
\end{lemma}

\begin{remark}
If the Euclidian space $\mathbb R^d$ is endowed with the usual spherical coordinates $(r, \theta) \in \mathbb R_+ \times \mathbb S^{d-1}$, 
the equilibrium measure $\nu$ reads
\begin{equation}
 \nu(r, \theta) = \frac{1}{Z} \times \frac{e^{- V(r) }}{\sigma^2(r)} r^{d-1} d r d \theta = \frac{1}{Z} \times \mu(r)^{d-1} \frac{e^{- \beta G(r) }}{\sigma^2(r)} d r d \theta, 
\end{equation}
where $d\theta$ is the uniform measure on $\mathbb S^{d-1}$.
\end{remark}

\begin{example}
In the case of the ROUP and the diffusion of Dunkel and H\"anggi, the invariant measure is the J\"uttner distribution which is the equivalent of the classical Maxwell measure in the framework of special relativity: 
\begin{equation}
\nu(\mathbf p)= \frac{1}{Z} \times e^{- \beta \sqrt{1+|| \mathbf p ||^2}} d \mathbf p , \quad \emph{i.e.} \quad \nu(r, \theta) =  \frac{1}{Z}  e^{- \beta \sqrt{1+r^2}} r^{d-1} d r d \theta. 
\end{equation}
\end{example}

\begin{proof}
Fix a smooth, bounded test function $f : \mathbb R^d \to \mathbb R$. One easily check that 
$$\int \mathcal L_{\s^2} f (\mathbf p) \frac{e^{- U(\mathbf p) } }{\s^2(||\p||)}  d\p = \frac{1}{2\beta} \int \mathcal L f (\mathbf p) e^{- U( \mathbf p) } d \mathbf p  =0,$$ 
so that the measure $\nu$ is invariant for the process $(\mathbf p_t)_{t \geq 0}$. Under the hypotheses $(\mathcal H)$, one has 
$\min(1,r) \leq \mu(r) \leq \max(1,r)$, $\sigma(r)\geq \varepsilon$ for all $r \geq 0$ and $G(r) \geq \varepsilon r /2$ for $r$ sufficiently large so that 
$$\int \frac{e^{- U(\mathbf p) }}{\sigma^2(||\mathbf p||)} d \mathbf p  = \int \mu(r)^{d-1} \frac{e^{- \beta G(r) }}{\sigma^2(r)} d r d \theta <+\infty.$$
The measure $\nu$ being finite, the process $(\mathbf p_t)_{t \geq 0}$ is ergodic.
\end{proof}

The next theorem and corollary establish that the equilibrium measure $\nu$ satisfies a Poincar\'e inequality, so that the rate of convergence to equilibrium of the momentum process is exponential in $\mathbb L^2(\nu)-$norm. Here and in the sequel, $\nu f$ denote the integral of the function $f$ against $\nu$.

\begin{theorem}\label{the.main}
There exists a positive constant $c_2$ such that for all $f \in \mathbb L^2(\nu)$:
$$\textrm{var}_{\nu} (f):=\vert \vert f - \nu f\vert \vert_{\mathbb L^2(\nu)}^2 \leq  c_2 \int \Gamma(f,f) d\nu,$$
\end{theorem}

\begin{remark}\label{rem.L1}
Following the same method as in the proof of Theorem \ref{the.main} in Section \ref{sec.proof} below, it can be shown (see \cite{bakry}) that  the Poincar\'e inequality even holds in $\mathbb L^1(\nu)$, namely there exists a positive constant $c_1$ such that for all $g \in \mathbb L^1(\nu)$ with zero median, one has:
$$\vert \vert g \vert \vert_{\mathbb L^1(\nu)} \leq  c_1 \int |\nabla g | d\nu. $$
\end{remark}

From the above theorem, by the well known equivalence between Poincar\'e inequality and exponential decay of the $\mathbb L^2(\nu)-$distance to equilibrium, see for example Theorem 2.5.5 of \cite{logsob}, we deduce the corollary:

\begin{corollary}\label{cor.main}
Let $(t, \mathbf x_t, p^0_t, \mathbf p_t)_{t \geq 0}$ be a diffusion of the class $\mathcal C$ and let $P_t$  be the Markov semi-group associated to the momentum subdiffusion $(\mathbf p_t)_{t \geq 0}$. Then, there exists a positive constant $c_2$ such that for all $t \geq 0$, and for all $f \in \mathbb L^2(\nu)$:
$$\vert \vert P_t f - \nu f\vert \vert_{\mathbb L^2(\nu)}^2 \leq  e^{-t/c_2 }\vert \vert f - \nu f\vert \vert_{\mathbb L^2(\nu)}^2.$$
\end{corollary}

Moreover, as it is shown in \cite{guillin}, from the exponential decay of the $\mathbb L^2(\nu)-$ distance to equilibrium, one can derive the exponential decay of the distance to equilibrium in total variation. Namely, if $P^*_t \widetilde{\nu}$ denotes the law of $\p_t$ with initial distribution $\widetilde{\nu}$, we have:

\begin{corollary} \label{cor.deuz}
Suppose that the initial distribution of $(\mathbf p_t)_{t \geq 0}$ can be written $\widetilde{\nu}=h d\nu$,  with $h \in \mathbb L^2(\nu)$. Then, for all $t \geq 0$, the distance in total variation satifies 
$$ \vert \vert P^*_t \widetilde{\nu} - \nu \vert \vert_{TV} = \vert \vert P^*_t h - 1 \vert \vert_{\mathbb L^1(\nu)} \leq  e^{-t/2c_2 }\vert \vert h - 1 \vert \vert_{\mathbb L^2(\nu)}.$$
\end{corollary}

The proof of Theorem \ref{the.main} is given in the next section. As told in the introduction, the method we follow is an adaptation of the method developped in \cite{bakry,bakry2}. It consists in expliciting a Lyapounov function for the generator $\mathcal L_{\sigma^2}$ and then use this function to derive a Poincar\'e inequality for the equilibrium measure $\nu$. The slight change from the original proof comes from the fact that in their original paper \cite{bakry}, Bakry et al. consider a framework where there is only an additive noise, \emph{i.e.} $\sigma\equiv 1$ and $\eta \equiv 0$. For the sake of completeness,  a self-contained proof of the theorem in our context is given below.

\begin{remark}
The exponential decay in Corollary \ref{cor.main} is not trivial, even in the simplest examples of diffusions belonging to the class $\mathcal C$. 
For example, in the case of the ROUP, the process $(\mathbf p_t)_{t \geq 0}$ is solution of the stochastic differential equations system:
$$ d \p_t = - \frac{\p_t}{\sqrt{1+||\p_t||^2}} dt + \sqrt{2/\beta} \, d  \mathbf W_t.$$
When $||\p_t ||$ is large, the drift is only linear in $t$ so that the amount of time needed to return to zero may be very large compared to the case of the usual Ornstein-Uhlenbeck process where the drift is linear in $\p_t$. Equivalently, from an analytic point of view, the potential 
$$U(\p) =V(||\p||)=\sqrt{1+||\p||^2}$$
is not strictly convex, so that the classical methods such as the Bakry-\'Emery criterion do not apply. 
However, such a result is not surprising since the exponential decay in $\mathbb L^2-$norm is well known (see \cite{bobkov}) for the very similar one dimensional process $X_t$ solution of the stochastic differential equation 
$$d X_t = -\textrm{sign}(X_t) dt + d W_t,$$
where the associated potential $x \mapsto |x|$ is not stricly convex either. In the last decade, much progress has been made ​​to weaken the assumptions under which a measure in Euclidean space satisfies a Poincar\'e inequality. Among recent advances, the method developped in  \cite{bakry,bakry2} provides an elegant and efficient way to get strong results, for example it applies to general log$-$concave measures.
\end{remark}

\section{Proof of the results} \label{sec.proof}

\subsection{Existence of a Lyapounov function for the generator $\mathcal L_{\s^2}$}
 We now give the proof of the theorem \ref{the.main} stated in the previous section. 
We first explicit a Lyapounov function associated to the infinitesimal generator $\mathcal L_{\s^2}$ given by the equation (\ref{eq.gener2}). We denote by $B(0,R)$ the Euclidian ball centered at the origin and of radius $R$ in $\mathbb R^d$. 

\begin{lemma}\label{lem.lya}
There exists a smooth function $W : \mathbb  R^d \to \mathbb R$, and some constants $\alpha>0$, $\gamma \geq 0$, $R>0$ such that for all $\mathbf p \in \mathbb R^d$:
\begin{enumerate}
\item $ W(\mathbf p) \geq 1$ ; \par
\vspace{0.2cm}
\item $\left | \frac{\nabla W}{W}(\mathbf p) \right| $  is bounded ; \par 
\vspace{0.2cm}
\item $ \mathcal L_{\s^2} W(\mathbf p) \leq - \alpha \,  W(\mathbf p) + \gamma \, 1_{B(0,R)}(\mathbf p)$.
\end{enumerate}
\end{lemma}
\begin{proof}
Consider a smooth function $W$ of the form $e^{c ||\mathbf p ||}$ for $||\mathbf p ||\geq R$ and such that $W(\mathbf p) \geq 1$ for all $\mathbf p \in \mathbb R^d$, where the two constants $c>0$ and $R>0$ will be fixed later. For $||\mathbf p ||\geq R$, one has
$$\mathcal L_{\s^2} W(\mathbf p) =\frac{\s^2(||\p||)}{2\beta} \times c \left( \frac{d-1}{||\mathbf p||} + c - \nabla U(\mathbf p). \frac{\p}{||\p||} \right) W(\mathbf p).$$
In our case, we have $\nabla U(\mathbf p) = \nabla V(||\mathbf p||) = V'(||\p||) .\p/|| \p||$ so that 
$$\mathcal L_{\s^2} W(\mathbf p) =\frac{\s^2(||\p||)}{2\beta} \times c \left( \frac{d-1}{||\mathbf p||} + c - V'(||\p||)\right) W(\mathbf p).$$
Under the hypotheses $(\mathcal H)$, $\s^2(||\p||) \geq \e^2$ for all $\p$, and for large enough $||\p||$ we have: 
$$V'(||\p||) =  \frac{d-1}{||\p||}\times \frac{\eta^2(||\p||)}{1+\eta(||\p||)^2}  + \beta g(||\p||) \geq \beta \e .$$
Thus, taking $c$ sufficiently small and $R$ large enough so that $\frac{d-1}{R}+c \leq \frac{\beta \e}{2}$, we get for $||\p||\geq R$:
$$\mathcal L_{\s^2} W(\mathbf p)  \leq -\alpha  W(\mathbf p), \;\; \hbox{where} \;\; \alpha= c \times \frac{\e^3}{4}.$$
Finally,  for some non-negative constant $\gamma$, we have for all $\p \in \mathbb R^d$:
$$\mathcal  L_{\s^2} W(\mathbf p)  \leq -\alpha  W(\mathbf p)+ \gamma \, 1_{B(0,R)}(\mathbf p).$$
Moreover, by construction $\left | \frac{\nabla W}{W}(\mathbf p) \right| $ is bounded. 
\end{proof}
\begin{remark}\label{rem.mono}
Since the Lyapounov function $W$ satisfies $W(\p)\geq 1$ for all $\p \in \mathbb R^d$, the point 3. in Lemma \ref{lem.lya} can be re-written as:
\begin{equation} 1 \leq - \frac{1}{\alpha} \frac{\mathcal  L_{\s^2}(W)}{W} + \frac{\gamma}{\alpha} \frac{1}{W}1_{B(0,R)}.\label{inq}\end{equation}
Therefore, since $W(|| \p ||)$ goes to infinity with $|| \p || $, the two ratios  $-\frac{\mathcal L_{\s^2}(W)}{W}$ and $-\frac{\mathcal L(W)}{W}$
are positive for $|| \p || $ large enough.
\end{remark}

\subsection{Proof of the Poincar\'e inequality}
We now give the proof of the Poincar\'e inequality stated in Theorem \ref{the.main}. In the sequel $d \lambda$ denote the Lebesgue measure in $\mathbb R^d$.
\begin{proof}
If $g$ is a smooth function in $\mathbb L^2(\nu)$, we have $\textrm{var}_{\nu}(g) \leq \int (g-c)^2 d \nu$ for all real constants $c$. Let $c$ be such a constant and define $f:=g-c$. Using the inequality (\ref{inq}), we have
\begin{equation} \label{ineq.var}
\int f^2 d \nu \leq - \frac{1}{\alpha} \underbrace{\int \frac{\mathcal  L_{\s^2}(W)}{W} f^2 d \nu}_{A} + \frac{\gamma}{\alpha} \underbrace{\int \frac{f^2}{W} 1_{B(0,R)} d \nu}_B.
\end{equation}
The first term A could be infinite depending on the behavior of $\mathcal L_{\s^2}(W)/W$ at infinity. Since we do not impose any integrability condition on this ratio, we have to restrict ourself in the calculation below to the case where $f=(g-c) \chi$, where $\chi$ is a smooth, non-negative, compactly supported function such that $1_{B(0,R)} \leq \chi \leq 1$. The general case is then obtained by taking a sequence of functions $\chi_n$ such that $1_{B(0,nR)} \leq \chi_n \leq 1$, $\vert \nabla \chi_n \vert \leq 1$, and go to the limit, which is allowed thanks to the mononicity noticed at the end of Remark \ref{rem.mono}.
Since the generator $\mathcal L$ is symmetric with respect to $e^{-U}d \lambda$, we have
$$
\begin{array}{ll}
A & \ds{= \int \frac{\mathcal  L_{\s^2}(W)}{W} f^2 d \nu =  \frac{1}{2\beta Z}  \int \frac{\mathcal  L(W)}{W} f^2 e^{-U}d \lambda} \\
&  = \ds{ - \frac{1}{2\beta Z}  \int \nabla \left(\frac{ f^2}{W} \right)\nabla W e^{-U}d \lambda } \\
& =\ds{-\frac{1}{2\beta Z}  \left(2 \int f \nabla f \frac{ \nabla W}{W}  e^{-U}d \lambda - \int f^2 \left \vert \frac{\nabla W}{W} \right \vert^2 e^{-U}d \lambda\right) } \\
& = \ds{-\frac{1}{2\beta Z} \left( \int  \vert \nabla f \vert^2 e^{-U}d \lambda - \int \left \vert \nabla f - f\frac{\nabla W}{W} \right \vert^2 e^{-U}d \lambda \right)},
\end{array}
$$
and therefore 
\begin{equation} \label{ineq.A}
A \geq - \frac{1}{2\beta Z}  \int  \vert \nabla f \vert^2 e^{-U}d \lambda.
\end{equation}
Let us consider now the second term $B$ in the right hand side of (\ref{ineq.var}). It is well know that the measure $\nu$ satisfies a Poincar\'e inequality in the ball $B(0,R)$, \emph{i.e.} for a positive constant $\kappa_R$, we have: 
$$\int_{B(0,R)} f^2 d \nu \leq  \kappa_R \int_{B(0,R)} | \nabla f |^2 d \nu + \frac{1}{\nu(B(0,R))} \left(\int_{B(0,R)} f d \nu\right)^2.$$
We choose $c=\nu(B(0,R))^{-1} \times \int_{B(0,R)}g d \nu$, so that the last term vanishes ; thus using the fact that $|W| \geq 1$ and $\s^2 \geq \e^2$, we have 
\begin{equation} \label{ineq.B}
B \leq \kappa_R \int_{B(0,R)} | \nabla f |^2 d \nu \leq \frac{\kappa_R}{ \e^2 Z}  \int_{B(0,R)} | \nabla f |^2 e^{-U}d \lambda.
\end{equation}
Putting (\ref{ineq.A}) and (\ref{ineq.B}) together, we conclude that 
$$ \begin{array}{ll}
\textrm{var}_{\nu}(g) & \ds{\leq \int f^2 d \nu \leq \left( \frac{1}{2\alpha \beta Z} + \frac{\gamma \kappa_R}{\alpha \e^2 Z} \right) \int|\nabla  f|^2 e^{-U}d \lambda} \\
& =  \ds{\left( \frac{1}{2\alpha \beta Z} + \frac{\gamma  \kappa_R}{\alpha \e^2 Z} \right) \int|\nabla  g|^2 e^{-U}d \lambda} \\
& = \ds{2\beta Z \left( \frac{1}{2\alpha \beta Z} + \frac{\gamma  \kappa_R}{\alpha \e^2 Z} \right) \int \frac{\s^2}{2 \beta} |\nabla  g|^2 d \nu} \\
& = \ds{ 2\beta Z \left( \frac{1}{2\alpha \beta Z} + \frac{\gamma  \kappa_R}{\alpha \e^2 Z} \right) \int \Gamma(g,g) d \nu}
\end{array}
 $$
In other words, we have shown the desired Poincar\'e inequality with the (very non optimal) Poincar\'e constant 
$$c_2:=  \frac{1}{\alpha }  \left( 1+ \frac{2 \beta \gamma  \kappa_R}{ \e^2 } \right).$$

\end{proof}

\bigskip
\noindent 
{\bf Acknowledgements}:
The author wishes to thank S. Calogero for the stimulating discussion during the meeting  ANR ProbaGeo in Strasbourg, and F. Malrieu for having informed him of the Lyapounov method.

\bibliographystyle{spmpsci}      
\bibliography{references}   

\end{document}